\documentclass[11pt]{amsart}

\usepackage[draft]{changes}

\definechangesauthor[color=red]{zl}
\definechangesauthor[color=blue]{hh}
\definechangesauthor[color=cyan]{hx}
\newcommand{\stkout}[1]{\ifmmode\text{\sout{\ensuremath{#1}}}\else\sout{#1}\fi}
\setdeletedmarkup{\stkout{#1}}

\usepackage{amssymb} 

\usepackage{graphics}
\usepackage{graphicx}

\usepackage{amsmath} 
\usepackage{amsthm}
\usepackage{amsfonts}
\usepackage{xcolor}
\usepackage[english]{babel}
\usepackage[margin=1.0in]{geometry}
\usepackage[colorlinks=true,
        linkcolor=blue]{hyperref}
\usepackage{bbm}
\usepackage{verbatim}

\usepackage[pagewise]{lineno}

\usepackage[T1]{fontenc}

\numberwithin{equation}{section}

\newtheorem{prop}{Proposition}
\newtheorem{lemma}[prop]{Lemma}

\newtheorem{thm}[prop]{Theorem}

\numberwithin{prop}{section}

\newtheorem{defn}[prop]{Definition}
\theoremstyle{definition}

\newtheorem{rmk}[prop]{Remark}

\definecolor{c1}{rgb}{0.2,0.4,0.5}
\definecolor{c2}{rgb}{0.1,0.3,0.5}
\definecolor{c3}{rgb}{0.2,0.7,0.5}
\usepackage{tikz}

\def \k {K\"ahler }
\newcommand{\oo}[1]{\overline{#1}}

\newcommand{\del}{\partial}
\newcommand{\bdel}{\bar{\partial}}

\newcommand{\gw}{\omega}

\newcommand{\ten}{\otimes}

\DeclareMathOperator{\Real}{Re}
\DeclareMathOperator{\Imaginary}{Im}

\DeclareMathOperator{\Hess}{Hess}

\begin{document}

\title[On a property of Bergman kernels when the K\"ahler potential is analytic]{On a property of Bergman kernels when the K\"ahler potential is analytic}

\begin{abstract} We provide a simple proof of a result of Rouby-Sj\"ostrand-Ngoc \cite{RSN} and Deleporte \cite{Deleporte}, which asserts that if the K\"ahler potential is real analytic then the Bergman kernel is an \textit{analytic kernel} meaning that its amplitude is an \textit{analytic symbol} and its phase is given by the polarization of the \k potential. This in particular shows that in the analytic case the Bergman kernel accepts an asymptotic expansion in a fixed neighborhood of the diagonal with an exponentially small remainder.  The proof we provide is based on a linear recursive formula of L. Charles \cite{Cha03} on the Bergman kernel coefficients which is similar to, but simpler than, the ones found in \cite{BBS}.
\end{abstract}


\author [Hezari]{Hamid Hezari}
\address{Department of Mathematics, University of California Irvine, Irvine, CA 92697, USA} \email{\href{mailto:hezari@uci.edu}{hezari@uci.edu}}

\author [Xu]{Hang Xu}
\address{Department of Mathematics, University of California San Diego, La Jolla, CA 92093 , USA}
\email{\href{mailto:h9xu@ucsd.edu}{h9xu@ucsd.edu}}

\maketitle

\section{Introduction}

Let $(L,h) \to M$ be a positive Hermitian holomorphic line bundle over a compact complex manifold of dimension $n$. The metric $h$ induces the \k form $\gw= -\tfrac{\sqrt{-1}}{2} \del \bdel \log(h)$ on $M$.  For $k$ in $\mathbb N$, let $H^0(M,L^k)$ denote the space of holomorphic sections of $L^k$. The {Bergman projection} is the orthogonal projection $\Pi_k: {L}^{2}(M,L^k) \to H^0(M,L^k)$ with respect to the natural inner product induced by the metric $h^k$ and the volume form $\frac{ \gw^n }{n!}$. The \emph{Bergman kernel} $K_k$, a section of $L^k\ten \bar{L}^k$, is the distribution kernel of $\Pi_k$.
Given $p \in M$,  let $(V, e_L)$ be a local trivialization of $L$ near $p$.   We write $| e_L |^2_{h}=e^{-\phi}$ and call $\phi$ a local \k potential. In the frame $ e_L^{k} \ten {\bar{e}_L^{k}}$, the Bergman kernel $K_k(x,y)$ is understood as a function on $V \times V$. 

In the smooth case, Zelditch \cite{Ze1} and Catlin \cite{Ca} proved that on the diagonal $x=y$, the Bergman kernel accepts a complete asymptotic expansion, as $k \to \infty$, of the form
\begin{equation}\label{ZC}
K_k(x,x)e^{-k\phi(x)}\sim \left(\frac{k}{\pi} \right)^n \left(b_0(x,\bar{x})+\frac{b_1(x,\bar{x})}{k}+\frac{b_2(x,\bar{x})}{k^2}+\cdots\right).
\end{equation}
Off-diagonal asymptotic expansion was given by Shiffman-Zelditch\cite{ShZe} and Charles \cite{Cha03}. For the analytic case, Zelditch conjectured \cite{Zeemail} that there is an off-diagonal asymptotic expansion for the Bergman kernel in a fixed neighborhood of the diagonal with an exponentially small remainder term. In particular as a result of this, one has
$$  \left | K_k(x, y) \right |_{h^k} \sim \left(\frac{k}{\pi} \right)^n  e^{- \frac{k D(x, y)}{2}}, \quad \text{as} \quad k \to \infty,$$ uniformly for all $x,y$ with $d(x, y) \leq \delta$ for some $\delta>0$. Here, $D(x, y)$ is Calabi's diastasis function (see \cite{Cal}) defined by
\begin{equation}\label{Diastatis} 
D(x, y)= \phi(x) + \phi(y) - \psi(x, \bar y) - \psi(y, \bar x),  
\end{equation}
which is controlled from above and below by $d^2(x, y)$, i.e. the square of the distance function induced by the \k metric.  This conjecture was proved first  by \cite{RSN} and shortly after a different proof was given by \cite{Deleporte}.

The main goal of this paper is to provide a new proof of this result. Roughly speaking, it states that the Bergman kernel is an \textit{analytic kernel} in the sense that it can be written as 
$$ K_k(x,y)=\left(\frac{k}{\pi} \right)^n e^{k\psi(x,\bar y)} b(x, \bar{y}, k),$$
where $\psi(x, \bar y)$ is an analytic phase function and $b(x, \bar{y}, k)$ is an \textit{analytic symbol}, meaning that $b(x,\bar{y}, k)$ has an asymptotic expansion (as $k \to \infty$) of the form 
\begin{equation} \label{analytic symbol} b(x,\bar{y}, k) \sim \sum_{m=0}^\infty \frac{b_m(x, \bar{y})}{k^m}, \end{equation} such that each $b_m(x, \bar{y})$ is holomorphic in a neighborhood of the diagonal $x=y$ and satisfies
$$ |b_m(x, \bar{y}) | \leq A^{m+1} m!, $$
for some $A>0$ independent of $m$. More precisely, we have: 
\begin{thm}\label{Main}
 Assume that the local \k potential $\phi$ is real analytic in $V$. Let $\psi(x, \bar{y})$ be the holomorphic extension of $\phi(x)$ near the diagonal obtained by \textit{polarization}, i.e., $\psi(x,\bar{y})$ is holomorphic in $x$, anti-holomorphic in $y$, and $\psi(x,\bar{x})=\phi(x)$.
Then there exist an open set $U \subset V$ containing $p$, an analytic symbol $b(x, \bar{y}, k)$ as in \eqref{analytic symbol} defined on $U \times U$, and positive constants $\delta$ and $C$, such that uniformly for any $x,y\in U$, we have
\begin{equation*}
K_k(x,y)=\left(\frac{k}{\pi} \right)^ne^{k\psi(x,\bar y)} \left (1+\sum _{m=1}^{ [\frac{k}{C}]}\frac{b_m(x, \bar y)}{k^m}\right )+e^{k\left(\frac{\phi(x)}{2}+\frac{\phi(y)}{2}\right)}e^{-\delta k}O(1). 
\end{equation*} 
\end{thm}

An equivalent formulation of the above theorem is:

\begin{thm}\label{complete asymptotics}
Given the same assumptions and notations as in the above theorem, there exist positive  constants $C$ and $\delta$, and an open set $U \subset V$ containing $p$, such that for all $k$ and $N\in \mathbb{N}$, we have for all $x,y\in U$ satisfying $d(x,y)\leq \delta$,
\begin{equation}\label{complete asymptotics eq}
K_k(x,y)=\left(\frac{k}{\pi} \right)^n e^{k\psi(x,\bar{y})}\left(1+\sum_{m=1}^{{N-1}}\frac{b_m(x,\bar{y})}{k^m}+\mathcal{R}_N(x, \bar y,k)\right),
\end{equation}
where 
\begin{equation*}
	\left|\mathcal{R}_N(x, \bar y ,k)\right|\leq {\frac{C^{N}N!}{k^N}}.
\end{equation*}
\end{thm}

A key step in our proof is to show that the holomorphic extensions $b_m(x,\bar{y})$ of the Bergman kernel coefficients $b_m(x, \bar x)$ appearing in the on-diagonal expansion \eqref{ZC} of Zelditch \cite{Ze1} and Catlin \cite{Ca} form an analytic symbol. More precisely, we show that:
\begin{thm}\label{MainLemma}
Assume the \k potential $\phi$ is real analytic in some neighborhood $V$ of $p$. Then, there exists a neighborhood $U\subset V$ of $p$, such that for any $m \in \mathbb N$ we have
\begin{equation*}
\|b_m(x,\bar{y})\|_{L^\infty(U\times U)} \leq A^{m+1}m!,
\end{equation*}
where $A$ is a constant independent of $m$. 
\end{thm}
Our proof is based on a linear recursive formula of L. Charles (See \cite{Cha00} and also equation (20) of \cite{Cha03}) for $b_m$. In our work with Z. Lu \cite{HLXanalytic}, we used a linear recursive formula of \cite{BBS} and only obtained the bounds of size $m!^2$ for the Bergman coefficients.  However  the formula of \cite{Cha00, Cha03} is simpler than \cite{BBS} as it does not involve any implicit functions and is very explicit. It is important to point out that another explicit recursive formula for Bergman coefficients was provided earlier by M. Engli\v s (See (3.10) of \cite{En}) based on the idempotent property of the Bergman projection (thus the recursive formula of \cite{En} is quadratic). The works of Loi \cite{Loi} and Xu \cite{Xu} on computation of the Bergman coefficients are based on this recursive formula of Engli\v s.   

There is a huge literature on Bergman kernels on compact complex manifolds. Before we conclude the introduction we list some related work that were not cited above. The book of Ma and Marinescu \cite{MaMaBook} contains an introduction to the asymptotic expansion of the Bergman kernel and its applications. See also the book review \cite{ZeBookReview} for more on the applications of Bergman kernels.  Two original and inspirational papers in the subject are \cite{Fe, BoSj}. For on-diagonal and near-diagonal asymptotic expansions see \cite{Ze1, Ca, DLM, LuSh, MaMaOff, HKSX}. For computations of the coefficients see \cite{Lu, En, LuTian, Loi, Xu, HLXanalytic}. Off-diagonal decay properties are studied in \cite{Del, MaMaAgmon, Ch, Ch2, Ch3, HX, Ze3}.  For the constant curvature case, which is a special case of analytic potentials, see \cite{Ber, Liu, HLXanalytic, Deleporte2}. Some applications of the Bergman kernel, and of the closely related Szeg\"o kernel, can be found in \cite{Do, BSZ, ShZe, YZ}. For Bergman kernels in the more general setting of symplectic manifolds see \cite{ShZe, MaMa}.

\section{The local reproducing property of Bergman kernel}\label{sec BBS}
Throughout this paper, we assume that $\phi$ is real analytic in a small open neighborhood $V\subset M$ of a given point $p$. Let $B^n(0,r)$ be the ball of radius $r$ in $\mathbb C^n$. We identify $p$ with $0\in \mathbb{C}^n$ and $V$ with the ball $B^n(0,3)\subset\mathbb{C}^n$ and denote $U=B^n(0, 1)$. Let $e_L$ be a local holomorphic frame of $L$ over $V$ as introduced in the introduction. 
For each positive integer $k$, we denote $H_{k\phi}(V)$ to be the inner product space of $L^2$-holomorphic functions on $U$ with respect to
\begin{equation*}
\left ( u, v \right)_{k\phi}=\int_V u \bar v \, e^{-k\phi}dV,
\end{equation*} 
where $dV=\frac{\omega^n}{n!}$ is the natural volume form induced by the \k form $\omega= \frac{\sqrt{-1}}{2} \partial \bar \partial \phi$. 
So the norm of $u \in H_{k\phi}(V)$ is given by
\begin{equation*}
\|u\|^2_{k\phi}=\int_V |u|^2e^{-k\phi}dV. 
\end{equation*}

Since $\phi(x)$ is analytic in $V=B^n(0, 3)$, without the loss of generality we assume that the radius of convergence of it power series in terms of $x$ and $\bar x$ is $3$.  By replacing $\phi(x)$ by $\phi(x) - \phi(0)$, we can assume that $\phi(0)=0$.  We  then denote $\psi(x,z)$ to be the holomorphic extension of $\phi(x)$ by replacing $\bar x$ with $z$. This procedure is called polarization. One can easily verify that $ \psi(x, z)$ satisfies the formal definition of holomorphic extension, namely
\begin{itemize}
	\item $\psi(x, z)$ is holomorphic in $B^n(0, 3) \times B^n(0, 3)$.
	\item $\psi(x, \bar x) = \phi(x)$. 
\end{itemize}   
Moreover, since $\phi(x)$ is real-valued, we have $ \oo{\psi(x, z)} = \psi( \bar z, \bar x)$.

Given $x\in U=B^n(0,1)$, let $\chi_x$ be a smooth cut-off function such that
\begin{equation*}
\chi_x(z)=
\begin{cases}
1 & z\in B^{n}(x,\frac{1}{2})\\
0 & z\notin B^{n}(x,\frac{3}{4}).
\end{cases}.
\end{equation*}

We denote the off-diagonal asymptotic expansion up to order $N$ by
\begin{align}\label{asymptotic expression}
K_k^{(N)}(x,y)=\left(\frac{k}{\pi}\right)^n e^{k\psi(x,\bar{y})}\left(1+b_1(x,\bar{y})+\frac{b_2(x,\bar{y})}{k^2}+\cdots+\frac{b_N(x,\bar{y})}{k^N}\right),
\end{align}
 where $b_m(x,z)$ are the holomorphic extensions of $b_m(x,\bar{x})$ in \eqref{ZC}. \cite{BBS} (see Proposition 2.5) proves that $K_k^{(N)}(x,y)$ is a local Bergman kernel mod $O\left(k^{n-N-1}\right)$ in the following sense.

\begin{prop}
	For any holomorphic function $u\in H_{k\phi}(V)$, we have that for any $x\in U$,
	\begin{align}\label{local reproducing of the asymptotic}
		u(x)=\int_V\chi_x(y)u(y)K_k^{(N)}(x,y)e^{-k\phi(y)}dV_y+O\left(k^{n-N-1}\right)e^{k\phi(x)/2}\|u\|_{k\phi}
	\end{align}
\end{prop}
As we shall see, this local reproducing property is a key in obtaining a proof of Charles's recursive formula. 

\begin{rmk}\label{smooth BBS remark}
	In fact, \cite{BBS} (see Proposition 2.7) proves the above proposition for general smooth \k metrics. For the general case, $\psi(x,z)$ is any almost holomorphic extension of the \k potential $\phi$. That is, $D_{(\bar{x},\bar{z})}\psi(x,z)$ vanishes to infinite order when $z=\bar{x}$,  and $\psi(x,\bar{x})=\phi(x)$.
	Each $b_m(x,z)$ is an almost holomorphic extension of $b_m(x,\bar{x})$.
\end{rmk}

\section{The recursive formula of Charles}
In this section, our main goal is to prove the following recursive formula of L. Charles \cite{Cha03} on the Bergman kernel coefficients $b_m$. This formula is simpler than the formulas in \cite{BBS} because it does not contain any implicit functions obtained from change of coordinates. 

\begin{thm}\label{recursive formula theorem}
	For any $w\in U$, the Bergman kernel coefficient $b_m$ satisfies
	\begin{equation}\label{recursive formula}
	b_m(w,\bar{w})=-\frac{1}{v(w,\bar{w})}\sum_{j=1}^m\sum_{\nu-\mu=j}\sum_{2\nu\geq 3\mu}(-1)^{\mu}\frac{1}{\mu!\nu!}\Delta^{\nu}\left(S_w^{\mu}b_{m-j}(w,\bar{y})v(y,\bar{y})\right)\Big|_{y=w}.
	\end{equation} 
	Here $\Delta=\sum g^{i\bar{j}}(w)\partial_{y_i}\partial_{\oo{y_j}}$ is the Laplace operator whose coefficients are frozen at $w$, $S_w$ is the defined by 
	\begin{equation}\label{Sw}
		S_w(z)
		=\sum_{|\alpha|\geq 1, |\beta|\geq 1, |\alpha+\beta|\geq 3}\frac{D_z^{\alpha}D_{\bar{z}}^{\beta}\phi(w)}{\alpha!\beta!}(z-w)^{\alpha}\overline{(z-w)}^{\beta},
	\end{equation} 
	and $v=\det(g_{i\bar{j}})$ is the determinant of the \k metric.
\end{thm}

\begin{proof}[Proof of Theorem \ref{recursive formula}]
	By the explicit expression of the asymptotic expansion \eqref{asymptotic expression}, we can rewrite the local reproducing property \eqref{local reproducing of the asymptotic} into:
	\begin{equation}\label{local reproducing 1}
		u(x)=\left(\frac{k}{\pi}\right)^n\int_V \chi_x(y)u(y)e^{k(\psi(x,\bar{y})-\phi(y))}B^{(N)}(x,\bar{y})dV_y+O(k^{n-N-1})e^{k\phi(x)/2}\|u\|_{k\phi},
	\end{equation} 
	where
	\begin{equation*}
		B^{(N)}(x,\bar{y}):=1+\frac{b_1(x,\bar{y})}{k}+\cdots+\frac{b_N(x,\bar{y})}{k^N}.
	\end{equation*}
	For any $w\in B^n(0,1)$, we can write the Taylor expansion of $\phi$ at $w$ as
	\begin{equation*}
		\phi(z)=\sum_{\alpha,\beta\geq 0}\frac{D_z^{\alpha}D_{\bar{z}}^{\beta}\phi(w)}{\alpha!\beta!}(z-w)^{\alpha}\overline{(z-w)}^{\beta}, \quad \mbox{ for any } z\in B^n(w,1).
	\end{equation*}
	Let $f(z)$ be the holomorphic part (up to a constant) of this Taylor expansion.
	\begin{equation*}
		f(z)=\sum_{\alpha\geq 0}\frac{D_z^{\alpha}\phi(w)}{\alpha!}(z-w)^{\alpha}-\frac{\phi(w)}{2}.
	\end{equation*}
	Then we define a new \k potential $\widetilde{\phi}_w$ as
	\begin{equation}\label{definition of the phase}
		\widetilde{\phi}_w(z)=\phi(z)-f(z)-\overline{f(z)}=\sum_{|\alpha|, |\beta|\geq 1}\frac{D_z^{\alpha}D_{\bar{z}}^{\beta}\phi(w)}{\alpha!\beta!}(z-w)^{\alpha}\overline{(z-w)}^{\beta}.
	\end{equation}
	Clearly, $\widetilde{\phi}_w$ satisfies
	\begin{align}\label{taylor expansion of the phase}
		\widetilde{\phi}_w(z)=\sum_{i,j=1}^n\phi_{i \bar j}(w)(z_i-w_i)\overline{(z_j-w_j)}+O(|z-w|^3).
	\end{align}
	What's more, the holomorphic extension of $\widetilde{\phi}_w$ satisfies
	\begin{align*}
		\widetilde{\psi}_w(y,\bar{z})=\psi(y,\bar{z})-f(y)-\overline{f(z)}.
	\end{align*}
	Since the local reproducing property \eqref{local reproducing 1} holds for any homomorphic function $u\in H_{k\phi}(V)$, we can replace $u(y)$ by $u(y)e^{k(f(y)-f(x))}$ and obtain
	\begin{align*}
		u(x)=\left(\frac{k}{\pi}\right)^n\int_V\chi_x(y)u(y)e^{k(\widetilde{\psi}_w(x,\bar{y})-\widetilde{\phi}_w(y))}B^{(N)}(x,\bar{y})dV_y+O(k^{n-N-1})e^{k\widetilde{\phi}_w(x)/2}\|u\|_{k\widetilde{\phi}_w}.
	\end{align*}
	In particular, if we take $x=w$, then $\widetilde{\psi}_w(w,\bar{y})=0$ by \eqref{definition of the phase} and we have
	\begin{align}\label{local reproducing 2}
		u(w)=\left(\frac{k}{\pi}\right)^n\int_V\chi_w(y)u(y)e^{-k\widetilde{\phi}_w(y)}B^{(N)}(w,\bar{y})dV_y+O(k^{n-N-1})\|u\|_{k\widetilde{\phi}_w}.
	\end{align}
	We shall now focus on the integral 
	\begin{equation*}
		I(k)=\int_V\chi_w(y)u(y)e^{-k\widetilde{\phi}_w(y)}B^{(N)}(w,\bar{y})dV_y.
	\end{equation*}
	 We first note that 
	\begin{equation*}
		dV_y=\det \left(\phi_{i\bar{j}}(y)\right)dV_{E,y}:=v(y,\bar{y})dV_{E,y},
		\end{equation*}
	where $dV_E$ be the standard Euclidean volume form.
	Thus the integral $I(k)$ writes into
	\begin{equation*}
		I(k)=\int_{\mathbb{C}^n}\chi_w(y)u(y)e^{-k\widetilde{\phi}_w(y)}B^{(N)}(w,\bar{y})v(y,\bar{y})dV_{E,y}.
	\end{equation*}
	By \eqref{taylor expansion of the phase} and after shrinking the size of the neighborhood $V$ if necessary, it follows immediately that the phase function $\widetilde{\phi}_w$ satisfies
	\begin{itemize}
		\item[(i)] $\widetilde{\phi}_w(w)=0$ and $\widetilde{\phi}_w(y)> 0$ for any $y\in B^n(w,1)\setminus\{w\}$ ;
		\item[(ii)] $\widetilde{\phi}_w(y)$ has a unique critical point $y=w$ in $B^n(w,1)$.
	\end{itemize}
To prepare for the stationary phase lemma, we first need to study the determinant of the real Hessian matrix of $\widetilde{\phi}_w$.
\begin{lemma}\label{lemma Hessian determinate relation}
	Given an opens set $U\subset \mathbb{C}^n$ and a point $p\in U$. Let $h:U\rightarrow \mathbb{R}$ be a smooth real-valued function such that at some $p\in U$, $\partial_{z_i}\partial_{z_j}h(p)=\partial_{\oo{z_i}}\partial_{\oo{z_j}}h(p)=0$ for any $1\leq i,j\leq n$. If we denote the real and complex Hessian matrices of $h$ by $\Hess_{\mathbb{R}}(h)$ and $\Hess_{\mathbb{C}}(h)$ respectively, then
	\begin{equation}\label{Hessian determinant relation}
		\det\left(\Hess_{\mathbb{R}}(h)\right)(p)=4^n\left|\det\left(\Hess_{\mathbb{C}}(h)\right)\right|^2(p).
	\end{equation}	
\end{lemma}
\begin{proof}
	For each $1\leq i\leq n$, we write $z_i=x_i+\sqrt{-1}y_i$ for $x_i,y_i\in \mathbb{R}$. Let $D_x=(D_{x_1},D_{x_2},\cdots,D_{x_n})$ be the gradient vector and let $D_x^{\intercal}$ be its transpose. We use similar notations for $D_y$, $D_z$ and $D_{\bar{z}}$.  Then it follows immediately that
	\begin{align}\label{matrix P}
	\begin{pmatrix}
	D_x^{\intercal} \\
	D_y^{\intercal}
	\end{pmatrix}
	=\begin{pmatrix}
	I_n & I_n\\
	\sqrt{-1}I_n & -\sqrt{-1}I_n
	\end{pmatrix}
	\begin{pmatrix}
	D_z^{\intercal}\\
	D_{\bar{z}}^{\intercal}
	\end{pmatrix}
	:=P\begin{pmatrix}
	D_z^{\intercal}\\
	D_{\bar{z}}^{\intercal}
	\end{pmatrix},	
	\end{align}
	and 
	\begin{align*}
	\left(D_x,D_y\right)=\left(D_z,D_{\bar{z}}\right)P^{\intercal}.
	\end{align*}
	If we take the matrix multiplication of the above two identities, then we obtain
	\begin{align}\label{Hessian matrix relation}
	\Hess_{\mathbb{R}}(h)=PBP^{\intercal}.
	\end{align}
	where $B$ is a $(2n)\times (2n)$ matrix defined by
	\begin{equation*}
		B=\begin{pmatrix}
		\left(h_{z_iz_j}\right)_{1\leq i,j\leq n} & \left(h_{z_i\oo{z_j}}\right)_{1\leq i,j\leq n}\\
		\left(h_{\oo{z_i}z_j}\right)_{1\leq i,j\leq n} & \left(h_{\oo{z_i}\oo{z_j}}\right)_{1\leq i,j\leq n}
		\end{pmatrix}.
	\end{equation*}
	By taking the determinant and evaluating it at $p$, we have
	\begin{equation*}
		\det\left(\Hess_{\mathbb{R}}(h)\right)(p)=\det\left(B\right)\det\left(P\right)^2=4^n\left|\det\left(\Hess_{\mathbb{C}}(h)\right)\right|^2(p).
	\end{equation*}	
\end{proof}

By applying the Lemma \ref{lemma Hessian determinate relation} to the phase function $h=\widetilde{\phi}_w$ and using \eqref{taylor expansion of the phase}, we have
\begin{itemize}
	\item[(iii)] $\det\left(\Hess_{\mathbb{R}}(\widetilde{\phi}_w)\right)(w)=4^n\left|\det\left(\phi_{z\bar{z}}(w)\right)\right|^2>0.$
\end{itemize}

Since the phase function $\widetilde{\phi}_w$ satisfies properties (i), (ii) and (iii), we can apply the stationary phase lemma (see Theorem 7.7.5 in \cite{Ho2}) to the integral $I(k)$. Therefore,
\begin{equation*}
	I(k)=\left(\det\left(k\Hess_{\mathbb{R}}(\widetilde{\phi}_w)/2\pi\right)\right)^{-\frac{1}{2}}\sum_{j=0}^N\frac{1}{k^j}L_j\left(\chi_w(y)u(y)B^{(N)}(w,\bar{y})v(y,\bar{y})\right)\Big|_{y=w}+O\left(\frac{1}{k^{N+1}}\right)\|u\|_{C^{2N+2}}.
\end{equation*} 
Here $L_j$ is a differential operators of order $2j$ acting on a smooth function $g(y)$ at $w$, by
\begin{equation}
	L_j(g(y))=\sum_{\nu-\mu=j}\sum_{2\nu\geq 3\mu}2^{-\nu}(-1)^\mu\frac{1}{\mu!\nu!}\left<\Hess_{\mathbb{R}}(\widetilde{\phi}_w)^{-1}\Big|_w\left(D_{\Real y}, D_{\Imaginary y}\right), \left(D_{\Real y}, D_{\Imaginary y}\right)\right>^{\nu}\left(S_w^{\mu}g\right)(w),
\end{equation}
where the function $S_w$ is
\begin{align*}
	S_w(y,\bar{y})
	=&\widetilde{\phi}_w(y)-\frac{1}{2}\left<\Hess_{\mathbb{R}}\widetilde{\phi}_w(w)\cdot\left(\Real y,\Imaginary y\right),\left(\Real y,\Imaginary y\right)\right>
	\\
	=&\widetilde{\phi}_w(y)-\sum_{i,j=1}^n\phi_{i\bar{j}}(w)(y_i-w_i)\oo{(y_j-w_j)}.
\end{align*}

Let us simplify the above expression for $I(k)$. First, note that the quadratic differential operator in $L_j$ is simply a multiple of the (complex) Laplace operator $\Delta= \sum g^{i\bar{j}}(w)\partial_{y_i}\partial_{\bar{y_j}}$ whose coefficients are frozen at $w$.
\begin{lemma}
	\begin{equation}\label{Laplace operator relation}
		\left<\Hess_{\mathbb{R}}(\widetilde{\phi}_w)^{-1}\big|_w\left(D_{\Real y}, D_{\Imaginary y}\right), \left(D_{\Real y}, D_{\Imaginary y}\right)\right>=2\Delta.
	\end{equation}
\end{lemma}
\begin{proof}
	Recall the relation between the real and complex Hessian matrices \eqref{Hessian matrix relation} and the matrix $P$ in \eqref{matrix P}. Here we have
	\begin{equation*}
		\Hess_{\mathbb{R}}(\widetilde{\phi}_w)(w)
		=P\begin{pmatrix}
		0 & \left(\phi_{i\bar{j}}(w)\right)_{1\leq i,j\leq n}\\
		\left(\phi_{\bar{i}j}(w)\right)_{1\leq i.j\leq n} & 0
		\end{pmatrix}P^{\intercal}.
	\end{equation*}
	Therefore, by the change of coordinates as in \eqref{matrix P}, we have 
	\begin{align*}
		\left<\Hess_{\mathbb{R}}(\widetilde{\phi}_w)^{-1}\big|_w\left(D_{\Real y}, D_{\Imaginary y}\right), \left(D_{\Real y}, D_{\Imaginary y}\right)\right>
		=&\left(D_y,D_{\bar{y}}\right)P^{\intercal}\Hess_{\mathbb{R}}(\widetilde{\phi}_w)^{-1}\big|_w P\left(D_y,D_{\bar{y}}\right)^{\intercal}
		\\
		=&\left(D_y,D_{\bar{y}}\right)
		\begin{pmatrix}
		0 &\left(\phi_{i\bar{j}}(w)\right)\\
		\left(\phi_{\bar{i}j}(w)\right) & 0
		\end{pmatrix}^{-1}
		\left(D_y,D_{\bar{y}}\right)^{\intercal}
		\\
		=&2 \sum g^{i\bar{j}}(w)\partial_{i}\partial_{\bar{j}}.
	\end{align*}
	The last equality follows from the fact that $\phi$ is a local \k potential. 
\end{proof}

Second, by \eqref{Hessian determinant relation} we can simplify the determinant in $I(k)$:
\begin{align*}
	\det\left(k\Hess_{\mathbb{R}}(\widetilde{\phi}_w)/2\pi\right)
	=\left(\frac{k}{\pi}\right)^{2n}\left|\det\phi_{z\bar{z}}(w)\right|^2=\left(\frac{k}{\pi}\right)^{2n}v^2(y,\bar{y}).
\end{align*}
So we can rewrite $I(k)$ as
\begin{align*}
	I(k)=&\left(\frac{\pi}{k}\right)^nv^{-1}\sum_{j=0}^N\frac{1}{k^j}\sum_{\nu-\mu=j}\sum_{2\nu\geq 3\mu}(-1)^{\mu}\frac{1}{\mu!\nu!}\Delta^{\nu}\left(S_w^{\mu}\chi_w(y)u(y)B^{(N)}(w,\bar{y})v(y,\bar{y})\right)\Big|_{y=w}+O(k^{-N-1})\left\|u\right\|_{C^{2N+2}}
	\\
	=&\left(\frac{\pi}{k}\right)^nv^{-1}\sum_{m=0}^N\sum_{i+j=m}\frac{1}{k^m}\sum_{\nu-\mu=j}\sum_{2\nu\geq 3\mu}(-1)^{\mu}\frac{1}{\mu!\nu!}\Delta^{\nu}\left(S_w^{\mu}\chi_w(y)u(y)b_i(w,\bar{y})v(y,\bar{y})\right)\Big|_{y=w}+O(k^{-N-1})\left\|u\right\|_{C^{2N+2}}.
\end{align*}
We plug this back into \eqref{local reproducing 2} and obtain
\begin{align*}
	u(w)=v^{-1}\sum_{m=0}^N\sum_{i+j=m}\frac{1}{k^m}\sum_{\nu-\mu=j}\sum_{2\nu\geq 3\mu}(-1)^{\mu}&\frac{1}{\mu!\nu!}\Delta^{\nu}\left(S_w^{\mu}\chi_w(y)u(y)b_i(w,\bar{y})v(y,\bar{y})\right)\Big|_{y=w}
	\\
	&+O(k^{n-N-1})\left(\left\|u\right\|_{C^{2N+2}}+\|u\|_{k\widetilde{\phi}_w}\right).
\end{align*}
By comparing the coefficients of $k^{-m}$, we have that for any $m\geq 1$,
\begin{equation*}
	0=\sum_{i+j=m}\sum_{\nu-\mu=j}\sum_{2\nu\geq 3\mu}(-1)^{\mu}\frac{1}{\mu!\nu!}\Delta^{\nu}\left(S_w^{\mu}\chi_w(y)u(y)b_i(w,\bar{y})v(y,\bar{y})\right)\Big|_{y=w}
\end{equation*}
Since this identity works for any holomorphic function $u\in H_{k\phi}(V)$, for each $m$ the coefficient of $u(w)$ on the right side vanishes. Therefore, for any $m\geq 1$,
\begin{equation*}
	0=\sum_{i+j=m}\sum_{\nu-\mu=j}\sum_{2\nu\geq 3\mu}(-1)^{\mu}\frac{1}{\mu!\nu!}\Delta^{\nu}\left(S_w^{\mu}b_i(w,\bar{y})v(y,\bar{y})\right)\Big|_{y=w}.
\end{equation*}
Note that the term $b_m$ only shows up when $i=m$ and $j=0$. After moving $b_m$ to the other side we obtain the recursive formula:
\begin{equation*}
	b_m(w,\bar{w})=-\frac{1}{v}\sum_{j=1}^m\sum_{\nu-\mu=j}\sum_{2\nu\geq 3\mu}(-1)^{\mu}\frac{1}{\mu!\nu!}\Delta^{\nu}\left(S_w^{\mu}b_{m-j}(w,\bar{y})v(y,\bar{y})\right)\Big|_{y=w}.
\end{equation*} 
\end{proof}

\begin{rmk}
	The recursive formula \eqref{recursive formula} in fact holds for the case of smooth \k potentials, if $S_w$ is understood as a smooth function with the prescribed Taylor series at $w$ as in \eqref{Sw} and $b_{m-j}(w,\bar{y})$ in \eqref{recursive formula} is any almost holomorphic extension of $b_{m-j}(w,\bar{w})$.  The proof in the smooth case follows in a similar way as the case of real analytic \k potential and we only point out the difference here. When $\phi$ is a smooth \k potential on $V$, we can still write out the Taylor expansion of $\phi$ at $w\in B^n(0,1)$.
	\begin{align*}
		\phi(z)\sim\sum_{\alpha,\beta\geq 0}\frac{D_z^{\alpha}D_{\bar{z}}^{\beta}\phi(w)}{\alpha!\beta!}(z-w)^{\alpha}\overline{(z-w)}^{\beta}, \quad \mbox{ for any } z\in B^n(w,1).
	\end{align*}
	For a given integer $M\geq 2$, we can define the up to order $M$ holomorphic part in the Taylor expansion.
	\begin{equation*}
		f^{(M)}(z)=\sum_{0\leq |\alpha|\leq M}\frac{D_z^{\alpha}\phi(w)}{\alpha!}(z-w)^{\alpha}-\frac{\phi(w)}{2}.
	\end{equation*} 
	And further define a \k potential $\widetilde{\phi}_w^{(M)}$ as 
	\begin{equation*}
		\widetilde{\phi}_w^{(M)}=\phi(z)-f^{(M)}(z)-\oo{f^{(M)}(z)}.
	\end{equation*}
	As mentioned in Remark \ref{smooth BBS remark}, the results in Section \ref{sec BBS} are valid for general smooth \k potentials, when $\psi(x,z)$ and $b_j(x,z)$ are modified to be certain almost holomorphic extension of $\phi(x)$ and $b_j(x,\bar{x})$. Therefore, we can perform the same computations as in the real analytic case and obtain 
	\begin{equation}\label{recursive formula smooth}
	b_m(w,\bar{w})=-\frac{1}{v}\sum_{j=1}^m\sum_{\nu-\mu=j}\sum_{2\nu\geq 3\mu}(-1)^{\mu}\frac{1}{\mu!\nu!}\Delta^{\nu}\left((S_w^{(M)})^{\mu}b_{m-j}(w,\bar{y})v(y,\bar{y})\right)\Big|_{y=w},
	\end{equation}  
	where 
	\begin{equation*}
		S_w^{(M)}(y,\bar{y})=\widetilde{\phi}_w^{(M)}(y)-\sum_{i,j=1}^n\phi_{i\bar{j}}(w)(y_i-w_i)\oo{(y_j-w_j)}.
	\end{equation*}
	Note that \eqref{recursive formula smooth} holds for any $N\geq 2$ and the right hand side actually only depends on derivatives of $S_w^{(N)}$ and $b_{m-j}$ up to finite order at $(w,\bar{w})$. So we have the desired result for the smooth case.
\end{rmk}

\section{Estimates on Bergman Kernel Coefficients}\label{Sec ProofofMainLemma}
As before, we assume the \k metric is analytic in the neighborhood $V=B^n(0,3)$. We will estimate the growth rate of the Bergman kernel coefficients $b_m(x,z)$ as $m\rightarrow \infty$ for $x,z$ in the open set $U=B^n(0,1)$. Our goal is to prove Theorem \ref{MainLemma}. 

The key ingredient for the proof is the recursive formula \eqref{recursive formula}.
We will break the proof of Theorem \ref{MainLemma} into two steps. The first step is to derive from the recursive formula \eqref{recursive formula}, a recursive inequality on $|D^{\gamma}_wD^{\delta}_{\bar{w}}b_m(w,\bar{w})|$ for any multi-index $\gamma,\delta \in (\mathbb{Z}^{\geq 0})^n$ and a given $w\in U$. The second step is to estimate $|D^{\gamma}_wD^{\delta}_{\bar{w}}b_m(w,\bar{w})|$ by induction.

In the following, we will use Greek letters $\alpha,\beta,\gamma\cdots$, or a Greek letter with a lower index $\alpha_1, \alpha_2, \cdots$ to denote multi-indices in $(\mathbb{Z}^{\geq 0})^n$. On the other hand, we will use upper index to denote the different components in a multi-index, such as $\alpha=(\alpha^1,\alpha^2,\cdots, \alpha^n)$ or $\alpha_1=(\alpha_1^1,\alpha_1^2,\cdots, \alpha_1^n)$. We also use the following standard notations for multi-indicies. 

\begin{itemize}
\item $ \mathbbm 1 = (1, 1, \cdots, 1)$.
\item $|\alpha|=\alpha^1+\alpha^2+\cdots+\alpha^n$.
\item $ \alpha \leq \beta $ if $\alpha^1\leq \beta^1, \alpha^2\leq \beta^2,\cdots,\alpha^n\leq \beta^n$.
\item $\alpha< \beta$ if $ \alpha \leq \beta$ and $\alpha \ne \beta$.
\item $\binom{\alpha}{\beta}=\binom{\alpha^1}{\beta^1}\binom{\alpha^2}{\beta^2}\cdots \binom{\alpha^n}{\beta^n}$.
\item $ \alpha !=\alpha^1!\alpha^2!\cdots \alpha^n!$.
\item $\binom{\alpha}{\alpha_1,\alpha_2,\cdots,\alpha_k}=\frac{\alpha!}{\alpha_1!\alpha_2!\cdots\alpha_k!}$ for multi-indices $\alpha_1, \alpha_2,\cdots,\alpha_k$ and $\alpha$ such that $\alpha=\alpha_1+\alpha_2+\cdots+\alpha_k$.
\item $D_y^\alpha=D_{y_1}^{\alpha^1}D_{y_2}^{\alpha^2}\cdots D_{y_n}^{\alpha^n}$.
\end{itemize}

\begin{lemma} \label{recursive inequality lemma}
Suppose $\phi$ is real analytic in some neighborhood $V$ of $p$. Then there exist some positive constant $C$ and an open set $U\subset V$ containing $p$, such that for any non-negative integer $m$, multi-indices $\gamma,\delta\geq 0$ and any $w\in V$, we have
\begin{align}\label{recursive inequality}
\frac{\left|D_w^{\gamma}D_{\bar{w}}^{\delta}b_m\right|}{\gamma!\delta!}
\leq 
\sum_{j=1}^m\sum_{\nu=j}^{3j}j!
\sum_{\substack{\gamma_4\leq \gamma, \delta_4\leq \delta}}C^{|\gamma+\delta|-|\gamma_4+\delta_4|+j}
\sum_{|\alpha|=|\beta|=\nu} \sum_{\substack{\beta_2\leq \beta\\ |\beta_2|\leq j}}\frac{\left|D_w^{\gamma_4}D_{\bar{w}}^{\beta_2+\delta_4}b_{m-j}\right|}{\gamma_4!\beta_2!\delta_4!}.
\end{align}
\end{lemma}
\begin{proof}
We first work on the expression $\Delta^{\nu}\left(S_w^{\mu}b_{m-j}(w,\bar{y})v(y,\bar{y})\right)\Big|_{y=w}$ in the recursive formula \eqref{recursive formula}. We expand the operator $\Delta^{\nu}$ and obtain
\begin{align*}
	\Delta^{\nu}
	=\left(\sum_{i, j} g^{i\bar{j}}(w)\partial_i\partial_{\bar{j}}\right)^{\nu}
	=&\sum_{1\leq i_1, j_1, i_2, j_2, \cdots, i_{\nu}, j_{\nu}\leq n}g^{i_1\bar{j}_1}g^{i_2\bar{j}_2}\cdots g^{i_{\nu}\bar{j}_{\nu}}\partial_{i_1}\partial_{i_2}\cdots\partial_{i_{\nu}}\partial_{\bar{j}_1}\partial_{\bar{j}_2}\cdots\partial_{\bar{j}_{\nu}}
	\\
	=&\sum_{|\alpha|=|\beta|=\nu}\sum_{\substack{e_{i_1}+e_{i_2}+\cdots+e_{i_v}=\alpha, \\ e_{j_1}+e_{j_2}+\cdots+e_{j_v}=\beta }} g^{i_1\bar{j}_1}g^{i_2\bar{j}_2}\cdots g^{i_{\nu}\bar{j}_{\nu}} D_y^{\alpha}D_{\bar{y}}^{\beta},
\end{align*}
where $\{e_j\}_{j=1}^n$ are the standard basis vectors in $\mathbb{R}^n$. For simplicity, we will denote
\begin{align*}
	I=(i_1, i_2, \cdots, i_{\nu}), \quad
	J=(j_1, j_2, \cdots, j_{\nu}),
\end{align*}
and
\begin{align*}
		g^{I\bar{J}}=g^{i_1\bar{j}_1}g^{i_2\bar{j}_2}\cdots g^{i_{\nu}\bar{j}_{\nu}}.
\end{align*}
For any multi-index $\alpha\in (\mathbb{Z}^{\geq 0})^n$ and non-negative integer $\nu$, we define the set 
\begin{equation*}
	A_{\alpha\beta\nu}=\left\{(I,J):
	\begin{array}{ll}
	e_{i_1}+e_{i_2}+\cdots+e_{i_v}=\alpha\\ e_{j_1}+e_{j_2}+\cdots+e_{j_v}=\beta
	\end{array} \right\}.
\end{equation*}
Clearly, since $1\leq i_k, j_k\leq n$ for every $k$, the sizes of these sets satisfy
\begin{equation}\label{cardinal of A}
	\# A_{\alpha\beta\nu}\leq n^{2\nu}.
\end{equation}
In terms of these new notations, $\Delta^{\nu}$ writes into:
\begin{equation*}
	\Delta^{\nu}
	=\sum_{|\alpha|=|\beta|=\nu}\sum_{A_{\alpha\beta\nu}} g^{I\bar{J}} D_y^{\alpha}D_{\bar{y}}^{\beta}.
\end{equation*}
Since $b_{m-j}(w,\bar{y})$ is antiholomorphic in $y$, 
\begin{align*}
	\Delta^{\nu}\left(S_w^{\mu}b_{m-j}(w,\bar{y})v(y,\bar{y})\right)\Big|_{y=w}
	=\sum_{|\alpha|=|\beta|=\nu}\sum_{A_{\alpha\beta\nu}} g^{I\bar{J}} \sum_{\beta_1+\beta_2=\beta}\binom{\beta}{\beta_1,\beta_2}D_y^{\alpha}D_{\bar{y}}^{\beta_1}\left(S_w^{\mu}v\right)D_{\bar{y}}^{\beta_2}b_{m-j}(w,\bar{w}).
\end{align*}
Substituting the above equation into the recursive formula \eqref{recursive formula}, we  obtain
\begin{align*}
	b_m=-\frac{1}{v}\sum_{j=1}^m\sum_{\nu-\mu=j}\sum_{2\nu\geq 3\mu}(-1)^{\mu}\frac{1}{\mu!\nu!}\sum_{|\alpha|=|\beta|=\nu}\sum_{A_{\alpha\beta\nu}} g^{I\bar{J}} \sum_{\beta_1+\beta_2=\beta}\binom{\beta}{\beta_1,\beta_2}D_y^{\alpha}D_{\bar{y}}^{\beta_1}\left(S_w^{\mu}v\right)D_{\bar{y}}^{\beta_2}b_{m-j}.
\end{align*} 
Note that all the holomorphic derivatives of $S_w$ vanish at $w$ by \eqref{Sw}. Thus, unless $D_y^{\alpha}D_{\bar{y}}^{\beta_1}\left(S_w^{\mu}v\right)=0$, we actually have $|\beta_1|\geq \mu$, whence
\begin{equation}\label{beta2 restriction}
	|\beta_2|=|\beta|-|\beta_1|\leq \nu-\mu=j.
\end{equation}

If we apply $D_w^{\gamma}D_{\bar{w}}^{\delta}$ on both sides and obtain a recursive formula for the derivatives of $b_m(w,\bar{w})$ as follows. 
\begin{align}\label{recursive formula on derivatives}
\begin{split}
	D_w^{\gamma}D_{\bar{w}}^{\delta}b_m
	=-\sum_{j=1}^m&\sum_{\nu-\mu=j}\sum_{2\nu\geq 3\mu}(-1)^{\mu}\frac{1}{\mu!\nu!}
	\sum_{\substack{\gamma_1+\gamma_2+\gamma_3+\gamma_4=\gamma \\ \delta_1+\delta_2+\delta_3+\delta_4=\delta}}\binom{\gamma}{\gamma_1,\gamma_2,\gamma_3,\gamma_4}\binom{\delta}{\delta_1,\delta_2,\delta_3,\delta_4}D_w^{\gamma_1}D_{\bar{w}}^{\delta_1}\left(\frac{1}{v}\right)
	\\
	&\cdot\sum_{|\alpha|=|\beta|=\nu}\sum_{A_{\alpha\beta\nu}} D_w^{\gamma_2}D_{\bar{w}}^{\delta_2}g^{I\bar{J}} \sum_{\beta_1+\beta_2=\beta}\binom{\beta}{\beta_1,\beta_2}D_w^{\gamma_3}D_{\bar{w}}^{\delta_3}D_y^{\alpha}D_{\bar{y}}^{\beta_1}\left(S_w^{\mu}v\right)D_w^{\gamma_4}D_{\bar{w}}^{\beta_2+\delta_4}b_{m-j},
\end{split}
\end{align}
where $D_w^{\gamma_3}D_{\bar{w}}^{\delta_3}D_y^{\alpha}D_{\bar{y}}^{\beta_1}\left(S_w^{\mu}v\right)=D_w^{\gamma_3}D_{\bar{w}}^{\delta_3}\left(D_y^{\alpha}D_{\bar{y}}^{\beta_1}\left(S_w^{\mu}v\right)(w,\bar{w})\right)$.

We need the following two lemmas to estimate some factors in the above recursive formula.
\begin{lemma}\label{factor 1}
	There exists a positive constant $C$ such that for any multi-indices $\alpha,\beta\geq 0$ and any $w\in B^n(0,1)$, we have
	\begin{align*}
		&\left|D_w^{\alpha}D_{\bar{w}}^{\beta}\left(\frac{1}{v}\right)(w,\bar{w})\right|\leq C^{|\alpha+\beta|+1}\alpha!\beta!,
		\\
		&\left|D_w^{\alpha}D_{\bar{w}}^{\beta}g^{I\bar{J}}(w)\right|\leq C^{|\alpha+\beta|+\nu}\alpha!\beta!.
	\end{align*}
\end{lemma}

\begin{proof}
	In this proof, we will use $C$ to denote a constant which only depends on the dimension $n$ and the \k potential $\phi$ on $V=B^n(0,3)$. The first inequality follows directly from the fact that $v=\det\left(\phi_{i \bar j}\right)$ is positive and real analytic on $B^n(0,3)$. Similarly, for any $1\leq i, j\leq n$ and any $w\in B^n(0,1)$, we also have
	\begin{equation}\label{g upper bound}
		\left|D_w^{\alpha}D_{\bar{w}}^{\beta}\left(g^{i\bar{j}}\right)(w)\right|\leq C^{|\alpha+\beta|+1}\alpha!\beta!.
	\end{equation}
	To prove the second inequality, we recall the notion of \emph{majorant}.
	
	\begin{defn}[Majorant]
	Suppose $f(x),g(x)$ are two smooth functions defined nearby $x=w\in \mathbb{C}^n$. We say that $g$ is a majorant of $f$ at $w$, denoted as $f<<_w g$, if for any multi-indices $\alpha,\beta\geq 0$, we have
	\begin{equation*}
		\left|D_x^{\alpha}D_{\bar{x}}^{\beta}f(w)\right|\leq D_x^{\alpha}D_{\bar{x}}^{\beta}g(w). 
	\end{equation*}
	\end{defn}
	
	Using the notion of majorant, \eqref{g upper bound} means that for any $1\leq i, j\leq n$ and any $w\in B^n(0,1)$, we have
	\begin{equation*}
		g^{i\bar{j}}(y)<<_w \frac{C}{\prod_{k=1}^n(1-C(y_k-w_k))(1-C(\oo{y_k-w_k}))}. 
	\end{equation*}
	Clearly, by taking products, we have
	\begin{align*}
		g^{I\bar{J}}&=g^{i_1\bar{j}_1}g^{i_2\bar{j}_2}\cdots g^{i_{\nu}\bar{j}_{\nu}}
		\\
		&<<_w\frac{C^{\nu}}{\prod_{k=1}^n(1-C(y_k-w_k))^{\nu}(1-C(\oo{y_k-w_k}))^{\nu}}
		\\
		&=C^{\nu}\prod_{k=1}^n\left(1+\binom{\nu}{1}C(y_k-w_k)+\binom{\nu+1}{2}C^2(y_k-w_k)^2+\cdots\right)
		\\
		&\quad\quad\quad\quad \cdot\left(1+\binom{\nu}{1}C(\oo{y_k-w_k})+\binom{\nu+1}{2}C^2(\oo{y_k-w_k})^2+\cdots\right) .
	\end{align*}
	Therefore,
	\begin{align*}
		\left|D_y^{\alpha}D_{\bar{y}}^{\beta}g^{I\bar{J}}(w)\right|\leq C^{\nu+|\alpha+\beta|}\binom{(\nu-1)\mathbbm{1}+\alpha}{\alpha}\binom{(\nu-1)\mathbbm{1}+\beta}{\beta}, 
	\end{align*}
	where $\mathbbm{1}=(1,1,\cdots,1)\in \mathbb{R}^n$.
	The second inequality follows by the fact that for any index $\alpha$, 
	\begin{equation*}
		\binom{(\nu-1)\mathbbm{1}+\alpha}{\alpha}\leq 2^{\nu n+|\alpha|}.
	\end{equation*}
\end{proof}

\begin{lemma}\label{factor 2}
	There exists a positive constant $C$ such that for any multi-indices $\alpha,\beta,\gamma,\delta\geq 0$ and any $w\in B^n(0,1)$, we have
	\begin{align*}
	\left|D_w^{\gamma}D_{\bar{w}}^{\delta}\left(D_y^{\alpha}D_{\bar{y}}^{\beta}\left(S_w^{\mu}v\right)(w,\bar{w})\right)\right|\left(w,\bar{w}\right)\leq C^{|\alpha+\beta+\gamma+\delta|+\mu+1}\alpha!\beta!\gamma!\delta!.
	\end{align*}
\end{lemma}

\begin{proof}
	By a straightforward computation, we have
	\begin{align*}
		D_y^{\alpha}D_{\bar{y}}^{\beta}S_w^{\mu}(w,\bar{w})
		=\sum_{\substack{\alpha_1+\alpha_2+\cdots+\alpha_{\mu}=\alpha\\
		\beta_1+\beta_2+\cdots+\beta_{\mu}=\beta}}\binom{\alpha}{\alpha_1,\alpha_2,\cdots,\alpha_{\mu}}\binom{\beta}{\beta_1,\beta_2,\cdots,\beta_{\mu}}\prod_{k=1}^{\mu}D_y^{\alpha_k}D_{\bar{y}}^{\beta_k}S_w(w,\bar{w}).
	\end{align*}
	Recall that $S_w(y)$ defined in \eqref{Sw} contains no purely holomorphic or purely antihomomorphic terms in its Taylor series at $w$ and vanishes of third order at $w$. If we define 
	\begin{align*}
		B_{\alpha\beta\mu}=\left\{\{\alpha_j\}_{j=1}^{\mu},\{\beta_j\}_{j=1}^{\mu}: \begin{array}{ll}
		\alpha_1+\alpha_2+\cdots+\alpha_{\mu}=\alpha,\\
		\beta_1+\beta_2+\cdots+\beta_{\mu}=\beta,\\
		|\alpha_j|>0, |\beta_j|>0,
		|\alpha_j+\beta_j|\geq 3
		\end{array}\right\},
	\end{align*} 
	then
	\begin{align*}
		D_y^{\alpha}D_{\bar{y}}^{\beta}S_w^{\mu}(w,\bar{w})
		=\sum_{B_{\alpha\beta\mu}}\binom{\alpha}{\alpha_1,\alpha_2,\cdots,\alpha_{\mu}}\binom{\beta}{\beta_1,\beta_2,\cdots,\beta_{\mu}}\prod_{k=1}^{\mu}D_y^{\alpha_k}D_{\bar{y}}^{\beta_k}\phi(w).
	\end{align*}
	By taking more derivatives, we have
	\begin{align*}
		D_w^{\gamma}D_{\bar{w}}^{\delta}&\left(D_y^{\alpha}D_{\bar{y}}^{\beta}S_w^{\mu}(w,\bar{w})\right)
		\\
		&=\sum_{C_{\gamma\delta\mu}}\sum_{B_{\alpha\beta\mu}}\binom{\gamma}{\gamma_1 \cdots \gamma_{\mu}}\binom{\delta}{\delta_1 \cdots \delta_{\mu}}\binom{\alpha}{\alpha_1 \cdots \alpha_{\mu}}\binom{\beta}{\beta_1 \cdots \beta_{\mu}}\prod_{k=1}^{\mu}D_y^{\alpha_k+\gamma_k}D_{\bar{y}}^{\beta_k+\delta_k}\phi(w),
	\end{align*}
	where
	\begin{align*}
	C_{\gamma\delta\mu}=\left\{\{\gamma_j\}_{j=1}^{\mu},\{\delta_j\}_{j=1}^{\mu}: \begin{array}{ll}
	\gamma_1+\gamma_2+\cdots+\gamma_{\mu}=\gamma,\\
	\delta_1+\delta_2+\cdots+\delta_{\mu}=\delta	\end{array}\right\},
	\end{align*} 
	Since $\phi$ is real analytic on $B^n(0,3)$, there exists $C>0$, such that for any $\alpha,\beta,\gamma,\delta$ and any $w\in B^n(0,1)$, we have
	\begin{align*}
	&\left|D_w^{\gamma}D_{\bar{w}}^{\delta}\left(D_y^{\alpha}D_{\bar{y}}^{\beta}S_w^{\mu}(w,\bar{w})\right)\right|
	\\
	&\quad \leq \sum_{C_{\gamma\delta\mu}}\sum_{B_{\alpha\beta\mu}}\binom{\gamma}{\gamma_1 \cdots \gamma_{\mu}}\binom{\delta}{\delta_1 \cdots \delta_{\mu}}\binom{\alpha}{\alpha_1 \cdots \alpha_{\mu}}\binom{\beta}{\beta_1 \cdots \beta_{\mu}} C^{|\alpha+\beta+\gamma+\delta|+\mu}\prod_{k=1}^{\mu}(\alpha_k+\gamma_k)!(\beta_k+\delta_k)!
	\\
	&\quad \leq C^{|\alpha+\beta+\gamma+\delta|+\mu}\alpha!\beta!\gamma!\delta!\sum_{C_{\gamma\delta\mu}}\sum_{B_{\alpha\beta\mu}}2^{|\alpha+\beta+\gamma+\delta|}.
	\end{align*}
	The second inequality follows from the following elementary facts
	\begin{align*}
		\binom{\alpha_k+\gamma_k}{\alpha_k , \gamma_k}\leq 2^{|\alpha_k+\gamma_k|}, \quad \binom{\beta_k+\delta_k}{\beta_k , \delta_k}\leq 2^{|\beta_k+\delta_k|}.
	\end{align*}
	We can estimate the size of the set $B_{\alpha\beta\mu}$ as:
	\begin{align*}
		\# B_{\alpha\beta\mu}\leq \#C_{\alpha\beta\mu}\leq \binom{\alpha+(\mu-1) \mathbbm{1}}{(\mu-1)\mathbbm{1}}\binom{\beta+(\mu-1) \mathbbm{1}}{(\mu-1)\mathbbm{1}}\leq 2^{|\alpha+\beta|+2n\mu}.
	\end{align*} 
	Therefore,
	\begin{align*}
	\left|D_w^{\gamma}D_{\bar{w}}^{\delta}\left(D_y^{\alpha}D_{\bar{y}}^{\beta}S_w^{\mu}(w,\bar{w})\right)\right|
	\leq 2^{(4n+1)\mu}(4C)^{|\alpha+\beta+\gamma+\delta|}\alpha!\beta!\gamma!\delta!.
	\end{align*}
	The lemma follows by noting $v$ is real analytic in $B^n(0,3)$.
\end{proof}

Now we apply Lemma \ref{factor 1}, Lemma \ref{factor 2} and \eqref{beta2 restriction} to the recursive formula \eqref{recursive formula on derivatives} on derivatives of $b_m(w,\bar{w})$ and obtain that
\begin{align*}
	\left|D_w^{\gamma}D_{\bar{w}}^{\delta}b_m\right|
	\leq 
	&\sum_{j=1}^m\sum_{\nu-\mu=j}\sum_{2\nu\geq 3\mu}\frac{1}{\mu!\nu!}
	\sum_{\substack{\gamma_1+\gamma_2+\gamma_3+\gamma_4=\gamma \\ \delta_1+\delta_2+\delta_3+\delta_4=\delta}}\binom{\gamma}{\gamma_1,\gamma_2,\gamma_3,\gamma_4}\binom{\delta}{\delta_1,\delta_2,\delta_3,\delta_4}C^{|\gamma_1+\delta_1|+1}\gamma_1!\delta_1!
	\\
	&\sum_{|\alpha|=|\beta|=\nu}\sum_{A_{\alpha\beta\nu}} C^{|\gamma_2+\delta_2|+\nu}\gamma_2!\delta_2! \sum_{\substack{\beta_1+\beta_2=\beta\\ |\beta_2|\leq j}}\binom{\beta}{\beta_1,\beta_2}C^{|\gamma_3+\delta_3+\alpha+\beta_1|+\mu+1}\alpha!\beta_1!\gamma_3!\delta_3!\left|D_w^{\gamma_4}D_{\bar{w}}^{\beta_2+\delta_4}b_{m-j}\right|.
\end{align*}
After simplification, we have
\begin{align*}
&\frac{\left|D_w^{\gamma}D_{\bar{w}}^{\delta}b_m\right|}{\gamma!\delta!}\leq
\\
&\quad  \quad 
\sum_{j=1}^m\sum_{\nu-\mu=j}\sum_{2\nu\geq 3\mu}\frac{\nu!}{\mu!}
\sum_{\substack{\gamma_1+\gamma_2+\gamma_3+\gamma_4=\gamma \\ \delta_1+\delta_2+\delta_3+\delta_4=\delta}}C^{|\gamma+\delta|-|\gamma_4+\delta_4|+\mu+3\nu+2}
\sum_{|\alpha|=|\beta|=\nu}\sum_{A_{\alpha\beta\nu}}  \sum_{\substack{\beta_1+\beta_2=\beta\\ |\beta_2|\leq j}}\frac{\left|D_w^{\gamma_4}D_{\bar{w}}^{\beta_2+\delta_4}b_{m-j}\right|}{\gamma_4!\beta_2!\delta_4!}.
\end{align*}
By using the fact $\nu!/\mu!\leq 2^{\nu}j!$ and the upper bound \eqref{cardinal of A} on the size of $A_{\alpha\beta\mu}$, after renaming the constant $2nC^5$ to $C$, we obtain
\begin{align*}
\frac{\left|D_w^{\gamma}D_{\bar{w}}^{\delta}b_m\right|}{\gamma!\delta!}
 \leq 
\sum_{j=1}^m\sum_{\nu=j}^{3j}j!
\sum_{\substack{\gamma_1+\gamma_2+\gamma_3+\gamma_4=\gamma \\ \delta_1+\delta_2+\delta_3+\delta_4=\delta}}C^{|\gamma+\delta|-|\gamma_4+\delta_4|+\nu}
\sum_{|\alpha|=|\beta|=\nu} \sum_{\substack{\beta_1+\beta_2=\beta\\ |\beta_2|\leq j}}\frac{\left|D_w^{\gamma_4}D_{\bar{w}}^{\beta_2+\delta_4}b_{m-j}\right|}{\gamma_4!\beta_2!\delta_4!}.
\end{align*} 
Since for a given index $\gamma_4\leq \gamma$, 
\begin{align*}
	\#\{(\gamma_1,\gamma_2,\gamma_3): \gamma_1+\gamma_2+\gamma_3=\gamma-\gamma_4\}= \binom{\gamma-\gamma_4+2\cdot\mathbbm{1}}{2\cdot \mathbbm{1}}\leq 2^{|\gamma-\gamma_4|+2n},
\end{align*}
 the above recursive inequality can be further simplified into:
\begin{align*}
\frac{\left|D_w^{\gamma}D_{\bar{w}}^{\delta}b_m\right|}{\gamma!\delta!}
\leq 
\sum_{j=1}^m\sum_{\nu=j}^{3j}j!
\sum_{\substack{\gamma_4\leq \gamma, \delta_4\leq \delta}}(2C)^{|\gamma+\delta|-|\gamma_4+\delta_4|+3j+4n}
\sum_{|\alpha|=|\beta|=\nu} \sum_{\substack{\beta_2\leq \beta\\ |\beta_2|\leq j}}\frac{\left|D_w^{\gamma_4}D_{\bar{w}}^{\beta_2+\delta_4}b_{m-j}\right|}{\gamma_4!\beta_2!\delta_4!}.
\end{align*}
The result follows by renaming $(2C)^{4n+3}$ to $C$.
\end{proof}

Next we use this lemma to prove Theorem \ref{MainLemma}.

\subsection{Proof of Theorem \ref{MainLemma}}
We will argue by induction on $m$ and prove that for any integer $m\geq 0$, any multi-indices $\gamma,\delta\geq 0$ and any $w\in U=B^n(0,1)$,
\begin{equation}\label{upper bound of bm}
\left|D_w^{\gamma}D_{\bar{w}}^{\delta}b_m(w,\bar{w})\right| \leq\binom{m+|\delta|}{|\delta|} M^{m}(2C)^{|\gamma+\delta|}\gamma!\delta!m!,
\end{equation}
where $C$ is the same constant which appears on the right hand side of \eqref{recursive inequality} and $M$ is a bigger constant to be selected later. Note that the estimates of Theorem \ref{MainLemma} follow easily from \eqref{upper bound of bm} after writing  the Taylor series of $b_m(x, \bar{y})$ on the diagonal $y=x$ and noting that by \eqref{upper bound of bm} we have
\begin{align*}
\left|D_w^{\gamma}D_{\bar{w}}^{\delta}b_m(w, \bar{w})\right|
\frac{}{}
\leq C'^{|\gamma|+|\delta|+m}\gamma!\delta!m!,
\end{align*}
where $C'=4MC$. This in particular shows that the holomorphic extensions $b_m(x, \bar{y})$ exist on a uniform neighborhood of the diagonal $y=x$ whose size is controlled by $C'$ which is independent of $m$.

 Obviously \eqref{upper bound of bm} holds for $m=0$ and any $\gamma,\delta\geq 0$, since $b_0=1$ on $U$. Assume that \eqref{upper bound of bm} holds up to $m-1$ and we proceed to $m$. By \eqref{recursive inequality}, we have
\begin{align}\label{bm 1}
\begin{split}
	\frac{\left|D_w^{\gamma}D_{\bar{w}}^{\delta}b_m\right|}{(2C)^{|\gamma+\delta|}M^m\gamma!\delta!}
	\leq 
	\sum_{j=1}^m\sum_{\nu=j}^{3j}j!
	&\sum_{\substack{\gamma_4\leq \gamma, \delta_4\leq \delta}}2^{-|\gamma+\delta|+|\gamma_4+\delta_4|}(2C^2/M)^j
	\\
	&\cdot\sum_{|\alpha|=|\beta|=\nu} \sum_{\substack{\beta_2\leq \beta\\ |\beta_2|\leq j}} \binom{m-j+|\beta_2+\delta_4|}{|\beta_2+\delta_4|}\binom{\beta_2+\delta_4}{\beta_2}(m-j)!.
\end{split}
\end{align}
We have the combinatorial inequality 
\begin{equation*}
	\binom{\beta_2+\delta_4}{\beta_2}\leq \binom{|\beta_2+\delta_4|}{|\beta_2|},
\end{equation*}
and the combinatorial identity
\begin{equation*}
	\binom{m-j+|\beta_2+\delta_4|}{|\beta_2+\delta_4|}\binom{|\beta_2+\delta_4|}{|\beta_2|}
	=\binom{m-j+|\beta_2+\delta_4|}{|\delta_4|}\binom{m-j+|\beta_2|}{|\beta_2|}.
\end{equation*}
Observe that, since $|\beta_2|\leq j$ and $|\delta_4|\leq |\delta|$, we have
\begin{equation*}
\binom{m-j+|\beta_2+\delta_4|}{|\beta_2+\delta_4|}\binom{|\beta_2+\delta_4|}{|\beta_2|}
\leq \binom{m+|\delta|}{|\delta|}\binom{m}{j}.
\end{equation*}
Plugging this into \eqref{bm 1}, we obtain
\begin{align*}
	\frac{\left|D_w^{\gamma}D_{\bar{w}}^{\delta}b_m\right|}{(2C)^{|\gamma+\delta|}M^m\gamma!\delta!}
	\leq \binom{m+|\delta|}{|\delta|}m!
	\sum_{j=1}^m\sum_{\nu=j}^{3j}
	\sum_{\substack{\gamma_4\leq \gamma, \delta_4\leq \delta}}2^{-|\gamma+\delta|+|\gamma_4+\delta_4|}(2C^2/M)^j
	\sum_{|\alpha|=|\beta|=\nu} \#\{\beta_2:\beta_2\leq \beta\}.
\end{align*}
Clearly,
\begin{align*}
	\sum_{|\alpha|=|\beta|=\nu} \#\{\beta_2:\beta_2\leq \beta\}
	\leq \sum_{|\alpha|=|\beta|=\nu}2^{|\beta|}=\binom{\nu+n-1}{n-1}^22^{\nu}\leq 2^{3\nu+2n-2}.
\end{align*}
It follows that
\begin{align*}
\frac{\left|D_w^{\gamma}D_{\bar{w}}^{\delta}b_m\right|}{(2C)^{|\gamma+\delta|}M^m\gamma!\delta!}
&\leq \binom{m+|\delta|}{|\delta|}m!
\sum_{j=1}^m (2C^2/M)^j\sum_{\nu=j}^{3j}
 2^{3\nu+2n-2} \sum_{\gamma_4\leq \gamma}2^{-|\gamma|+|\gamma_4|} \sum_{ \delta_4\leq \delta}2^{-|\delta|+|\delta_4|}
\\
&\leq \binom{m+|\delta|}{|\delta|}m!
\sum_{j=1}^m (2C^2/M)^j\sum_{\nu=j}^{3j}
2^{3\nu+4n-2}
\\
&\leq \binom{m+|\delta|}{|\delta|}m!
2^{4n}\sum_{j=1}^m (2^{10}C^2/M)^j.
\end{align*}
By letting $M=2^{4n+11}C^2$, we  have $2^{4n}\sum_{j=1}^m (2^{10}C^2/M)^j\leq 1$ and
\begin{align*}
	\left|D_w^{\gamma}D_{\bar{w}}^{\delta}b_m\right|
	\frac{}{}
	\leq \binom{m+|\delta|}{|\delta|}m!(2C)^{|\gamma+\delta|}M^m\gamma!\delta!.
\end{align*}
Therefore the induction is concluded.

\subsection{Proof of Theorems \ref{Main} and \ref{complete asymptotics}} The proofs follow from Theorem \ref{MainLemma} in a non-trivial way. However these conclusions have already been carried out in \cite{HLXanalytic} (Remark 1.7).  

\section*{Acknowledgements}
In the final stage of the preparation of this paper S. Zelditch brought to our attention a recent arXiv preprint of L. Charles \cite{Ch19} which uses the Toeplitz calculus to prove that the Bergman kernel has an analytic symbol. Although the tools used in \cite{Ch19} are seemingly different from ours, the heart of the proof is similar as it uses a formula obtained from the reproducing property of the Bergman kernel (see Lemma 3.2 of \cite{Ch19}) . We are grateful to L. Charles and M. Engli\v s for their valuable comments on the literature. 


\end{document}